\documentclass[11pt]{article}
\usepackage{epsf}
\usepackage{amsmath}
\usepackage{epsfig}
\usepackage{times}
\usepackage{amssymb}
\usepackage{amsthm}
\usepackage{setspace}
\usepackage{cite}

\usepackage{algorithmic}  
\usepackage{algorithm}

\usepackage{shadow}
\usepackage{fancybox}
\usepackage{fancyhdr}

\def\y{{\bf y}}

\def\x{{\bf x}}

\def\x{{\mathbf x}}

\def\x{{\bf x}}
\def\y{{\bf y}}

\def\h{{\bf h}}

\def\cH{{\cal H}}

\def\be{\begin{equation}}
\def\ee{\end{equation}}
\def\ba{\left[\begin{array}}
\def\ea{\end{array}\right]}

\def\x{{\bf x}}
\def\y{{\bf y}}

\def\1{{\bf 1}}

\def\g{{\bf g}}
\def\0{{\bf 0}}

\newtheorem{theorem}{Theorem}

\newtheorem{lemma}{Lemma}

\begin{document}

\begin{singlespace}

\title {Bounding ground state energy of Hopfield models 
}
\author{
\textsc{Mihailo Stojnic}
\\
\\
{School of Industrial Engineering}\\
{Purdue University, West Lafayette, IN 47907} \\
{e-mail: {\tt mstojnic@purdue.edu}} }
\date{}
\maketitle

\centerline{{\bf Abstract}} \vspace*{0.1in}

In this paper we look at a class of random optimization problems that arise in the forms typically known as Hopfield models. We view two scenarios which we term as the positive Hopfield form and the negative Hopfield form. For both of these scenarios we define the binary optimization problems that essentially emulate what would typically be known as the ground state energy of these models. We then present a simple mechanism that can be used to create a set of theoretical rigorous bounds for these energies. In addition to purely theoretical bounds, we also present a couple of fast optimization algorithms that can also be used to provide solid (albeit a bit weaker) algorithmic bounds for the ground state energies.

\vspace*{0.25in} \noindent {\bf Index Terms: Hopfield models; ground-state energy}.

\end{singlespace}

\section{Introduction}
\label{sec:back}

We start by looking at what is typically known in mathematical physics as the Hopfield model. The model was popularized in \cite{Hop82} (or if viewed in a different context one could say in \cite{PasFig78,Hebb49}). It essentially looks at what is called Hamiltonian of the following type
\begin{equation}
\cH(H,\x)=\sum_{i\neq j}^{n}  A_{ij}\x_i\x_j,\label{eq:ham}
\end{equation}
where
\begin{equation}
A_{ij}(H)=\sum_{l=1}^{m}  H_{il}H_{lj},\label{eq:hamAij}
\end{equation}
are the so-called quenched interactions and $H$ is an $m\times n$ matrix that can be also viewed as the matrix of the so-called stored patterns (we will typically consider scenario where $m$ and $n$ are large and $\frac{m}{n}=\alpha$ where $\alpha$ is a constant independent of $n$; however, many of our results will hold even for fixed $m$ and $n$). Each pattern is essentially a row of matrix $H$ while vector $\x$ is a vector from $R^n$ that emulates neuron states. Typically, one assumes that the patterns are binary and that each neuron can have two states (spins) and hence the elements of matrix $H$ as well as elements of vector $\x$ are typically assumed to be from set $\{-\frac{1}{\sqrt{n}},\frac{1}{\sqrt{n}}\}$. In physics literature one usually follows convention and introduces a minus sign in front of the Hamiltonian given in (\ref{eq:ham}). Since our main concern is not really the physical interpretation of the given Hamiltonian but rather mathematical properties of such forms we will avoid the minus sign and keep the form as in (\ref{eq:ham}).

To characterize the behavior of physical interpretations that can be described through the above Hamiltonian one then looks at the partition function
\begin{equation}
Z(\beta,H)=\sum_{\x\in\{-\frac{1}{\sqrt{n}},\frac{1}{\sqrt{n}}\}^n}e^{\beta\cH(H,\x)},\label{eq:partfun}
\end{equation}
where $\beta>0$ is what is typically called the inverse temperature. Depending of what is the interest of studying one can then also look at a more appropriate scaled $\log$ version of $Z(\beta,H)$ (typically called the free energy)
\begin{equation}
f_p(n,\beta,H)=\frac{\log{(Z(\beta,H)})}{\beta n}=\frac{\log{(\sum_{\x\in\{-\frac{1}{\sqrt{n}},\frac{1}{\sqrt{n}}\}^n}e^{\beta\cH(H,\x)})}}{\beta n}.\label{eq:logpartfun}
\end{equation}
Studying behavior of the partition function or the free energy of the Hopfield model of course has a long history. Since we will not focus on the entire function in this paper we just briefly mention that a long line of results can be found in e.g. excellent references \cite{PasShchTir94,ShchTir93,BarGenGueTan10,BarGenGueTan12,Tal98}. In this paper though we will focus on studying optimization/algorithmic aspects of $\frac{\log{(Z(\beta,H)})}{\beta n}$. More specifically, we will look at a particular regime $\beta,n\rightarrow\infty$ (which is typically called a zero-temperature thermodynamic limit regime or as we will occasionally call it the ground state regime). In such a regime one has
\begin{equation}
\hspace{-.3in}\lim_{\beta,n\rightarrow\infty}f_p(n,\beta,H)=
\lim_{\beta,n\rightarrow\infty}\frac{\log{(Z(\beta,H)})}{\beta n}=\lim_{n\rightarrow\infty}\frac{\max_{\x\in\{-\frac{1}{\sqrt{n}},\frac{1}{\sqrt{n}}\}^n}\cH(H,\x)}{n}
=\lim_{n\rightarrow\infty}\frac{\max_{\x\in\{-\frac{1}{\sqrt{n}},\frac{1}{\sqrt{n}}\}^n}\|H\x\|_2^2}{n},\label{eq:limlogpartfun}
\end{equation}
which essentially renders the following form (often called the ground state energy)
\begin{equation}
\lim_{\beta,n\rightarrow\infty}f_p(n,\beta,H)=\lim_{n\rightarrow\infty}\frac{\max_{\x\in\{-\frac{1}{\sqrt{n}},\frac{1}{\sqrt{n}}\}^n}\|H\x\|_2^2}{n},\label{eq:posham}
\end{equation}
which will be one of the main subjects that we will study in this paper. We will refer to the optimization part of (\ref{eq:posham}) as the positive Hopfield form.

In addition to this form we will also study its a negative counterpart. Namely, instead of the partition function given in (\ref{eq:partfun}) one can look at a corresponding partition function of a negative Hamiltonian from (\ref{eq:ham}) (alternatively, one can say that instead of looking at the partition function defined for positive temperatures/inverse temperatures one can also look at the corresponding partition function defined for negative temperatures/inverse temperatures). In that case (\ref{eq:partfun}) becomes
\begin{equation}
Z(\beta,H)=\sum_{\x\in\{-\frac{1}{\sqrt{n}},\frac{1}{\sqrt{n}}\}^n}e^{-\beta\cH(H,\x)},\label{eq:partfunneg}
\end{equation}
and if one then looks at its an analogue to (\ref{eq:limlogpartfun}) one then obtains
\begin{equation}
\hspace{-.3in}\lim_{\beta,n\rightarrow\infty}f_n(n,\beta,H)=\lim_{\beta,n\rightarrow\infty}\frac{\log{(Z(\beta,H)})}{\beta n}=\lim_{n\rightarrow\infty}\frac{\max_{\x\in\{-\frac{1}{\sqrt{n}},\frac{1}{\sqrt{n}}\}^n}-\cH(H,\x)}{n}
=\lim_{n\rightarrow\infty}\frac{\min_{\x\in\{-\frac{1}{\sqrt{n}},\frac{1}{\sqrt{n}}\}^n}\|H\x\|_2^2}{n}.\label{eq:limlogpartfunneg}
\end{equation}
This then ultimately renders the following form which is in a way a negative counterpart to (\ref{eq:posham})
\begin{equation}
\lim_{\beta,n\rightarrow\infty}f_n(n,\beta,H)=\lim_{n\rightarrow\infty}\frac{\min_{\x\in\{-\frac{1}{\sqrt{n}},\frac{1}{\sqrt{n}}\}^n}\|H\x\|_2^2}{n}.\label{eq:negham}
\end{equation}
We will then correspondingly refer to the optimization part of (\ref{eq:negham}) as the negative Hopfield form.

In the following sections we will present a collection of results that relate to behavior of the forms given in (\ref{eq:posham}) and (\ref{eq:negham}) when they are viewed in a statistical scenario. The results that we will present will essentially correspond to what is called the ground state energies of these models. As it will turn out, in the statistical scenario that we will consider, (\ref{eq:posham}) and (\ref{eq:negham}) will be almost completely characterized by their corresponding average values
\begin{equation}
\lim_{\beta,n\rightarrow\infty}Ef_p(n,\beta,H)=\lim_{n\rightarrow\infty}\frac{E\max_{\x\in\{-\frac{1}{\sqrt{n}},\frac{1}{\sqrt{n}}\}^n}\|H\x\|_2^2}{n}\label{eq:poshamavg}
\end{equation}
and
\begin{equation}
\lim_{\beta,n\rightarrow\infty}Ef_n(n,\beta,H)=\lim_{n\rightarrow\infty}\frac{E\min_{\x\in\{-\frac{1}{\sqrt{n}},\frac{1}{\sqrt{n}}\}^n}\|H\x\|_2^2}{n}.\label{eq:neghamavg}
\end{equation}

Before proceeding further with our presentation we will be a little bit more specific about the organization of the paper. In Section \ref{sec:poshop} we will present a few results that relate to behavior of the positive Hopfield form in a statistical scenario. We will then in Section \ref{sec:neghop} present the corresponding results for the negative Hopfield form. In section \ref{sec:alghop} we will present several algorithmic results that will in a way complement our findings from Sections \ref{sec:poshop} and \ref{sec:neghop}. Finally, in Section \ref{sec:conc} we will give a few concluding remarks.

\section{Positive Hopfield form}
\label{sec:poshop}

In this section we will look at the following optimization problem (which clearly is the key component in estimating the ground state energy in the thermodynamic limit)
\begin{equation}
\max_{\x\in\{-\frac{1}{\sqrt{n}},\frac{1}{\sqrt{n}}\}^n}\|H\x\|_2^2.\label{eq:posham1}
\end{equation}
For a deterministic (given fixed) $H$ this problem is of course known to be NP-hard (it essentially falls under the class of binary quadratic optimization problems). Instead of looking at the problem in (\ref{eq:posham1}) in a deterministic way i.e. in a way that assumes that matrix $H$ is deterministic, we will look at it in a statistical scenario (this is of course a typical scenario in statistical physics). Within a framework of statistical physics and neural networks the problem in (\ref{eq:posham1}) is studied assuming that the stored patterns (essentially rows of matrix $H$) are comprised of Bernoulli $\{-1,1\}$ i.i.d. random variables see, e.g. \cite{Tal98,PasShchTir94,ShchTir93}. While our results will turn out to hold in such a scenario as well we will present them in a different scenario: namely, we will assume that the elements of matrix $H$ are i.i.d. standard normals. We will then call form (\ref{eq:posham1}) with Gaussian $H$, the Gaussian positive Hopfield form. On the other hand, we will call form (\ref{eq:posham1}) with Bernoulli $H$, the Bernoulli positive Hopfield form. In the remainder of this section we will look at possible ways to estimate the optimal value of the optimization problem in (\ref{eq:posham1}). In the first part below we will introduce a strategy that can be used to obtain an upper bound on the optimal value and in the second part we will then create a corresponding lower-bounding strategy.

\subsection{Upper-bounding ground state energy of the positive Hopfield form}
\label{sec:poshopub}

As we just mentioned above, in this section we will look at problem from (\ref{eq:posham1}). In fact, to be a bit more precise, in order to make the exposition as simple as possible, we will look at its a slightly changed version given below
\begin{equation}
\xi_p=\max_{\x\in\{-\frac{1}{\sqrt{n}},\frac{1}{\sqrt{n}}\}^n}\|H\x\|_2.\label{eq:sqrtposham1}
\end{equation}
As mentioned above, we will assume that the elements of $H$ are i.i.d. standard normal random variables. Before proceeding further with the analysis of (\ref{eq:sqrtposham1}) we will recall on several well known results that relate to Gaussian random variables and the processes they create.

We start by first recalling the following results from \cite{Gordon88} that relate to statistical properties of such Gaussian processes.
\begin{theorem}(\cite{Gordon88})
\label{thm:Gordonmesh1} Let $X_{ij}$ and $Y_{ij}$, $1\leq i\leq n,1\leq j\leq m$, be two centered Gaussian processes which satisfy the following inequalities for all choices of indices
\begin{enumerate}
\item $E(X_{ij}^2)=E(Y_{ij}^2)$
\item $E(X_{ij}X_{ik})\geq E(Y_{ij}Y_{ik})$
\item $E(X_{ij}X_{lk})\leq E(Y_{ij}Y_{lk}), i\neq l$.
\end{enumerate}
Then
\begin{equation*}
P(\bigcap_{i}\bigcup_{j}(X_{ij}\geq \lambda_{ij}))\leq P(\bigcap_{i}\bigcup_{j}(Y_{ij}\geq \lambda_{ij})).
\end{equation*}
\end{theorem}
The following, more simpler, version of the above theorem relates to the expected values.
\begin{theorem}(\cite{Gordon88})
\label{thm:Gordonmesh2} Let $X_{ij}$ and $Y_{ij}$, $1\leq i\leq n,1\leq j\leq m$, be two centered Gaussian processes which satisfy the following inequalities for all choices of indices
\begin{enumerate}
\item $E(X_{ij}^2)=E(Y_{ij}^2)$
\item $E(X_{ij}X_{ik})\geq E(Y_{ij}Y_{ik})$
\item $E(X_{ij}X_{lk})\leq E(Y_{ij}Y_{lk}), i\neq l$.
\end{enumerate}
Then
\begin{equation*}
E(\min_{i}\max_{j}(X_{ij}))\leq E(\min_i\max_j(Y_{ij})).
\end{equation*}
\end{theorem}
When $m=1$ both of the above theorems simplify to what is called Slepian's lemma (see, e.g. \cite{Slep62}). In fact, to be completely chronologically exact, the two above theorems actually extended the Slepian's lemma which, for the completeness, we include below in the form of two theorems that are effective analogues to Theorems \ref{thm:Gordonmesh1} and \ref{thm:Gordonmesh2}.
\begin{theorem}(\cite{Slep62,Gordon88})
\label{thm:Slepian1} Let $X_{i}$ and $Y_{i}$, $1\leq i\leq n$, be two centered Gaussian processes which satisfy the following inequalities for all choices of indices
\begin{enumerate}
\item $E(X_{i}^2)=E(Y_{i}^2)$
\item $E(X_{i}X_{l})\leq E(Y_{i}Y_{l}), i\neq l$.
\end{enumerate}
Then
\begin{equation*}
P(\bigcap_{i}(X_{i}\geq \lambda_{i}))\leq P(\bigcap_{i}(Y_{i}\geq \lambda_{i}))
\Leftrightarrow P(\bigcup_{i}(X_{i}\geq \lambda_{i}))\leq P(\bigcup_{i}(Y_{i}\geq \lambda_{i})).
\end{equation*}
\end{theorem}
The following, more simpler, version of the above theorem relates to the expected values.
\begin{theorem}(\cite{Slep62,Gordon88})
\label{thm:Slepian2} Let $X_{i}$ and $Y_{i}$, $1\leq i\leq n$, be two centered Gaussian processes which satisfy the following inequalities for all choices of indices
\begin{enumerate}
\item $E(X_{i}^2)=E(Y_{i}^2)$
\item $E(X_{i}X_{l})\leq E(Y_{i}Y_{l}), i\neq l$.
\end{enumerate}
Then
\begin{equation*}
E(\min_{i}(X_{i}))\leq E(\min_i(Y_{i})) \Leftrightarrow E(\max_{i}(X_{i}))\geq E(\max_i(Y_{i})).
\end{equation*}
\end{theorem}

Now, to create an upper-bounding strategy for the positive Hopfield form we will rely on Theorems \ref{thm:Slepian1} and Theorem \ref{thm:Slepian2}. We start by reformulating the problem in (\ref{eq:sqrtposham1}) in the following way
\begin{equation}
\xi_p=\max_{\x\in\{-\frac{1}{\sqrt{n}},\frac{1}{\sqrt{n}}\}^n}\max_{\|\y\|_2=1}\y^TH\x.\label{eq:sqrtposham2}
\end{equation}
We will first focus on the expected value of $\xi_p$ and then on its more general probabilistic properties. The following is then a direct application of Theorem \ref{thm:Slepian2}.
\begin{lemma}
Let $H$ be an $m\times n$ matrix with i.i.d. standard normal components. Let $\g$ and $\h$ be $m\times 1$ and $n\times 1$ vectors, respectively, with i.i.d. standard normal components. Also, let $g$ be a standard normal random variable. Then
\begin{equation}
E(\max_{\x\in\{-\frac{1}{\sqrt{n}},\frac{1}{\sqrt{n}}\}^n,\|\y\|_2=1}(\y^T H\x +\|\x\|_2 g))\leq E(\max_{\x\in\{-\frac{1}{\sqrt{n}},\frac{1}{\sqrt{n}}\}^n,\|\y\|_2=1}(\|\x\|_2\g^T\y+\h^T\x)).\label{eq:posexplemma}
\end{equation}\label{lemma:posexplemma}
\end{lemma}
\begin{proof}
As mentioned above, the proof is a standard/direct application of Theorem \ref{thm:Slepian2}. We will sketch it for completeness. Namely, one starts by defining processes $X_i$ and $Y_i$ in the following way
\begin{equation}
Y_i=(\y^{(i)})^T H\x^{(i)} +\|\x^{(i)}\|_2 g\quad X_i=\|\x^{(i)}\|_2\g^T\y^{(i)}+\h^T\x^{(i)}.\label{eq:posexplemmaproof1}
\end{equation}
Then clearly
\begin{equation}
EY_i^2=EX_i^2=2\|\x^{(i)}\|_2^2=2.\label{eq:posexplemmaproof2}
\end{equation}
One then further has
\begin{eqnarray}
EY_iY_l & = & (\y^{(i)})^T\y^{(l)}(\x^{(l)})^T\x^{(i)}+\|\x^{(i)}\|_2\|\x^{(l)}\|_2 \nonumber \\
EX_iX_l & = & (\y^{(i)})^T\y^{(l)}\|\x^{(i)}\|_2\|\x^{(l)}\|_2+(\x^{(l)})^T\x^{(i)}.\label{eq:posexplemmaproof3}
\end{eqnarray}
And after a small algebraic transformation
\begin{eqnarray}
EY_iY_l-EX_iX_l & = & \|\x^{(i)}\|_2\|\x^{(l)}\|_2(1-(\y^{(i)})^T\y^{(l)})-(\x^{(l)})^T\x^{(i)}(1-(\y^{(i)})^T\y^{(l)}) \nonumber \\
& = & (\|\x^{(i)}\|_2\|\x^{(l)}\|_2-(\x^{(l)})^T\x^{(i)})(1-(\y^{(i)})^T\y^{(l)})\nonumber \\
& \geq & 0.\label{eq:posexplemmaproof4}
\end{eqnarray}
Combining (\ref{eq:posexplemmaproof2}) and (\ref{eq:posexplemmaproof4}) and using results of Theorem \ref{thm:Slepian2} one then easily obtains (\ref{eq:posexplemma}).
\end{proof}

Using results of Lemma \ref{lemma:posexplemma} we then have
\begin{multline}
E(\max_{\x\in\{-\frac{1}{\sqrt{n}},\frac{1}{\sqrt{n}}\}^n} \|H\x\|_2) =E(\max_{\x\in\{-\frac{1}{\sqrt{n}},\frac{1}{\sqrt{n}}\}^n,\|\y\|_2=1}(\y^T H\x +\|\x\|_2 g))\\\leq E(\max_{\x\in\{-\frac{1}{\sqrt{n}},\frac{1}{\sqrt{n}}\}^n,\|\y\|_2=1}(\|\x\|_2\g^T\y+\h^T\x))=E\|\x\|_2\|\g\|_2+E\sum_{i=1}^{n}|\h_i|\leq \sqrt{m}+\sqrt{\frac{2}{\pi}}\sqrt{n}.\label{eq:poshopaftlemma2}
\end{multline}
Connecting beginning and end of (\ref{eq:poshopaftlemma2}) we finally have an upper bound on $E\xi_p$ from (\ref{eq:sqrtposham1}), i.e.
\begin{equation}
E\xi_p=E(\max_{\x\in\{-\frac{1}{\sqrt{n}},\frac{1}{\sqrt{n}}\}^n} \|H\x\|_2) \leq \sqrt{m}+\sqrt{\frac{2}{\pi}}\sqrt{n}=\sqrt{n}(\sqrt{\alpha}+\sqrt{\frac{2}{\pi}}),\label{eq:poshopubexp}
\end{equation}
or in a scaled (possibly) more convenient form
\begin{equation}
\frac{E\xi_p}{\sqrt{n}}=\frac{E(\max_{\x\in\{-\frac{1}{\sqrt{n}},\frac{1}{\sqrt{n}}\}^n} \|H\x\|_2)}{\sqrt{n}} \leq \sqrt{\alpha}+\sqrt{\frac{2}{\pi}}.\label{eq:poshopubexp1}
\end{equation}

We now turn to deriving a more general probabilistic result related to $\xi_p$. Before doing so we do mention that since the ground state energies will concentrate in thermodynamic limit (more on a much more general approach in this direction can be found in e.g. \cite{GiuGen12}), their expected values considered above are typically the hardest thing to study. In that regard the probabilistic results that we will present below may not be viewed as important. However, although here for the easiness of the exposition we often assume a large $n$ scenario many of the concepts that we present work just fine even for finite $n$. One should then keep in mind that the strategy we present below has an importance attached to it that goes beyond a likelihood type generalization of the above studied means.

Now, we will present this more general probabilistic estimate of the ground state energy through the following lemma.
\begin{lemma}
Let $H$ be an $m\times n$ matrix with i.i.d. standard normal components. Let $\g$ and $\h$ be $m\times 1$ and $n\times 1$ vectors, respectively, with i.i.d. standard normal components. Also, let $g$ be a standard normal random variable and let $\zeta_{\x}$ be a function of $\x$. Then
\begin{equation}
P(\max_{\x\in\{-\frac{1}{\sqrt{n}},\frac{1}{\sqrt{n}}\}^n\|\y\|_2=1}(\y^T H\x +\|\x\|_2 g-\zeta_{\x})\geq 0)\leq
P(\max_{\x\in\{-\frac{1}{\sqrt{n}},\frac{1}{\sqrt{n}}\}^n\|\y\|_2=1}(\|\x\|_2\g^T\y+\h^T\x-\zeta_{\x})\geq 0).\label{eq:posproblemma}
\end{equation}\label{lemma:posproblemma}
\end{lemma}
\begin{proof}
The proof is basically same as the proof of Lemma \ref{lemma:posexplemma}. The only difference is that instead of Theorem \ref{thm:Slepian2} it relies on Theorem \ref{thm:Slepian1}.
\end{proof}

Let $\zeta_{\x}=-\epsilon_{5}^{(g)}\sqrt{n}\|\x\|_2+\xi_p^{(u)}$ with $\epsilon_{5}^{(g)}>0$ being an arbitrarily small constant independent of $n$. We will first look at the right-hand side of the inequality in (\ref{eq:posproblemma}). The following is then the probability of interest
\begin{equation}
P(\max_{\x\in\{-\frac{1}{\sqrt{n}},\frac{1}{\sqrt{n}}\}^n\|\y\|_2=1}(\|\x\|_2\g^T\y+\h^T\x+\epsilon_{5}^{(g)}\sqrt{n}\|\x\|_2)\geq \xi_p^{(u)}).\label{eq:probanal0}
\end{equation}
After solving the maximization over $\x$ and $\y$ one obtains
\begin{equation}
\hspace{-.3in}P(\max_{\x\in\{-\frac{1}{\sqrt{n}},\frac{1}{\sqrt{n}}\}^n\|\y\|_2=1}(\|\x\|_2\g^T\y+\h^T\x+\epsilon_{5}^{(g)}\sqrt{n}\|\x\|_2)\geq \xi_p^{(u)})=P(\|\g\|_2+\sum_{i=1}^{n}|\h_i|/\sqrt{n}+\epsilon_{5}^{(g)}\sqrt{n}\geq \xi_p^{(u)}).\label{eq:probanal1}
\end{equation}
Since $\g$ is a vector of $m$ i.i.d. standard normal variables it is rather trivial that $P(\|\g\|_2<(1+\epsilon_{1}^{(m)})\sqrt{m})\geq 1-e^{-\epsilon_{2}^{(m)} m}$ where $\epsilon_{1}^{(m)}>0$ is an arbitrarily small constant and $\epsilon_{2}^{(m)}$ is a constant dependent on $\epsilon_{1}^{(m)}$ but independent of $n$. Along the same lines, since $\h$ is a vector of $n$ i.i.d. standard normal variables it is rather trivial that $P(|\h|<(1+\epsilon_{1}^{(n)})n)\geq 1-e^{-\epsilon_{2}^{(n)} n}$ where $\epsilon_{1}^{(n)}>0$ is an arbitrarily small constant and $\epsilon_{2}^{(n)}$ is a constant dependent on $\epsilon_{1}^{(n)}$ but independent of $n$. Then from (\ref{eq:probanal1}) one obtains
\begin{multline}
P(\max_{\x\in\{-\frac{1}{\sqrt{n}},\frac{1}{\sqrt{n}}\}^n\|\y\|_2=1}(\|\x\|_2\g^T\y+\h^T\x+\epsilon_{5}^{(g)}\sqrt{n}\|\x\|_2)\geq \xi_p^{(u)})\\\leq
(1-e^{-\epsilon_{2}^{(m)} m})(1-e^{-\epsilon_{2}^{(n)} n})
P((1+\epsilon_{1}^{(m)})\sqrt{m}+(1+\epsilon_{1}^{(n)})\sqrt{n}\sqrt{\frac{2}{\pi}}+\epsilon_{5}^{(g)}\sqrt{n}\geq \xi_p^{(u)})
+e^{-\epsilon_{2}^{(m)} m}+e^{-\epsilon_{2}^{(n)} n}.\label{eq:probanal2}
\end{multline}
If
\begin{eqnarray}
& & (1+\epsilon_{1}^{(m)})\sqrt{m}+(1+\epsilon_{1}^{(n)})\sqrt{n}\sqrt{\frac{2}{\pi}}+\epsilon_{5}^{(g)}\sqrt{n}<\xi_p^{(u)}\nonumber \\
& \Leftrightarrow & (1+\epsilon_{1}^{(m)})\sqrt{\alpha}+(1+\epsilon_{1}^{(n)})\sqrt{\frac{2}{\pi}}+\epsilon_{5}^{(g)}<\frac{\xi_p^{(u)}}{\sqrt{n}},\label{eq:condxipu}
\end{eqnarray}
one then has from (\ref{eq:probanal2})
\begin{equation}
\lim_{n\rightarrow\infty}P(\max_{\x\in\{-\frac{1}{\sqrt{n}},\frac{1}{\sqrt{n}}\}^n\|\y\|_2=1}(\|\x\|_2\g^T\y+\h^T\x+\epsilon_{5}^{(g)}\sqrt{n}\|\x\|_2)\geq \xi_p^{(u)})\leq 0.\label{eq:probanal3}
\end{equation}

We will now look at the left-hand side of the inequality in (\ref{eq:posproblemma}). The following is then the probability of interest
\begin{equation}
P(\max_{\x\in\{-\frac{1}{\sqrt{n}},\frac{1}{\sqrt{n}}\}^n\|\y\|_2=1}(\y^TH\x+\|\x\|_2g+\epsilon_{5}^{(g)}\sqrt{n}\|\x\|_2-\xi_p^{(u)})\geq 0).\label{eq:leftprobanal0}
\end{equation}
Since $P(g\geq -\epsilon_{5}^{(g)}\sqrt{n})\geq 1-e^{-\epsilon_{6}^{(g)} n}$ (where $\epsilon_{6}^{(g)}$ is, as all other $\epsilon$'s in this paper are, independent of $n$) from (\ref{eq:leftprobanal0}) we have
\begin{multline}
P(\max_{\x\in\{-\frac{1}{\sqrt{n}},\frac{1}{\sqrt{n}}\}^n\|\y\|_2=1}(\y^TH\x+\|\x\|_2g+\epsilon_{5}^{(g)}\sqrt{n}\|\x\|_2-\xi_p^{(u)})\geq 0)
\\\geq (1-e^{-\epsilon_{6}^{(g)} n})P(\max_{\x\in\{-\frac{1}{\sqrt{n}},\frac{1}{\sqrt{n}}\}^n\|\y\|_2=1}(\y^TH\x-\xi_p^{(u)})\geq 0).\label{eq:leftprobanal1}
\end{multline}
When $n$ is large from (\ref{eq:leftprobanal1}) we then have
\begin{multline}
\hspace{-.4in}\lim_{n\rightarrow \infty}P(\max_{\x\in\{-\frac{1}{\sqrt{n}},\frac{1}{\sqrt{n}}\}^n\|\y\|_2=1}(\y^TH\x+\|\x\|_2g+\epsilon_{5}^{(g)}\sqrt{n}\|\x\|_2-\xi_p^{(u)})\geq 0)
\geq  \lim_{n\rightarrow\infty}P(\max_{\x\in\{-\frac{1}{\sqrt{n}},\frac{1}{\sqrt{n}}\}^n\|\y\|_2=1}(\y^TH\x-\xi_p^{(u)})\geq 0)\\
 =  \lim_{n\rightarrow\infty}P(\max_{\x\in\{-\frac{1}{\sqrt{n}},\frac{1}{\sqrt{n}}\}^n\|\y\|_2=1}(\y^TH\x)\geq \xi_p^{(u)})
 =  \lim_{n\rightarrow\infty}P(\max_{\x\in\{-\frac{1}{\sqrt{n}},\frac{1}{\sqrt{n}}\}^n}(\|H\x\|_2)\geq \xi_p^{(u)}).\label{eq:leftprobanal2}
\end{multline}
Assuming that (\ref{eq:condxipu}) holds, then a combination of (\ref{eq:posproblemma}), (\ref{eq:probanal3}), and (\ref{eq:leftprobanal2}) gives
\begin{equation}
 \lim_{n\rightarrow\infty}P(\max_{\x\in\{-\frac{1}{\sqrt{n}},\frac{1}{\sqrt{n}}\}^n}(\|H\x\|_2)\geq \xi_p^{(u)})\leq \lim_{n\rightarrow\infty}P(\max_{\x\in\{-\frac{1}{\sqrt{n}},\frac{1}{\sqrt{n}}\}^n\|\y\|_2=1}(\|\x\|_2\g^T\y+\h^T\x+\epsilon_{5}^{(g)}\sqrt{n}\|\x\|_2)\geq \xi_p^{(u)})\leq 0.\label{eq:leftprobanal3}
\end{equation}

We summarize our results from this subsection in the following lemma.

\begin{lemma}
Let $H$ be an $m\times n$ matrix with i.i.d. standard normal components. Let $n$ be large and let $m=\alpha n$, where $\alpha>0$ is a constant independent of $n$. Let $\xi_p$ be as in (\ref{eq:sqrtposham1}). Let all $\epsilon$'s be arbitrarily small constants independent of $n$ and let $\xi_p^{(u)}$ be a scalar such that
\begin{equation}
(1+\epsilon_{1}^{(m)})\sqrt{\alpha}+(1+\epsilon_{1}^{(n)})\sqrt{\frac{2}{\pi}}+\epsilon_{5}^{(g)}<\frac{\xi_p^{(u)}}{\sqrt{n}}.\label{eq:condxipuposgenlemma}
\end{equation}
Then
\begin{eqnarray}
& & \lim_{n\rightarrow\infty}P(\max_{\x\in\{-\frac{1}{\sqrt{n}},\frac{1}{\sqrt{n}}\}^n}(\|H\x\|_2)\leq \xi_p^{(u)})\geq 1\nonumber \\
& \Leftrightarrow & \lim_{n\rightarrow\infty}P(\xi_p\leq \xi_p^{(u)})\geq 1 \nonumber \\
& \Leftrightarrow & \lim_{n\rightarrow\infty}P(\xi_p^2\leq (\xi_p^{(u)})^2)\geq 1, \label{eq:posgenproblemma}
\end{eqnarray}
and
\begin{equation}
\frac{E\xi_p}{\sqrt{n}}=\frac{E(\max_{\x\in\{-\frac{1}{\sqrt{n}},\frac{1}{\sqrt{n}}\}^n} \|H\x\|_2)}{\sqrt{n}} \leq \sqrt{\alpha}+\sqrt{\frac{2}{\pi}}.\label{eq:posgenexplemma}
\end{equation}
\label{lemma:posgenlemma}
\end{lemma}
\begin{proof}
The proof follows from the above discussion, (\ref{eq:poshopubexp1}), and (\ref{eq:leftprobanal3}).
\end{proof}

\subsection{Lower-bounding ground state energy of the positive Hopfield form}
\label{sec:poshoplb}

In this subsection we will create the corresponding lower-bound results. To create a lower-bounding strategy for the positive Hopfield form we will again (as in previous subsection) rely on Theorems \ref{thm:Slepian1} and \ref{thm:Slepian2}. We start by recalling that the problem of interest is the one in (\ref{eq:sqrtposham2}) and we rewrite it in the following way
\begin{equation}
\xi_p=\max_{\x\in\{-\frac{1}{\sqrt{n}},\frac{1}{\sqrt{n}}\}^n}\max_{\|\y\|_2=1}\y^TH\x.\label{eq:sqrtposham2lb}
\end{equation}
As in the previous subsection, we will first focus on the expected value of $\xi_p$ and then on its more general probabilistic properties. The following is then a direct application of Theorem \ref{thm:Slepian2}.
\begin{lemma}
Let $H$ be an $m\times n$ matrix with i.i.d. standard normal components. Let $H^{(1)}$ and $H^{(2)}$ be $m\times m$ and $n\times n$ matrices, respectively, with i.i.d. standard normal components. Then
\begin{equation}
E(\max_{\x\in\{-\frac{1}{\sqrt{n}},\frac{1}{\sqrt{n}}\}^n,\|\y\|_2=1}(\y^T H\x))\geq E(\max_{\x\in\{-\frac{1}{\sqrt{n}},\frac{1}{\sqrt{n}}\}^n,\|\y\|_2=1}(\frac{1}{\sqrt{2}}\y^TH^{(1)}\y+\frac{1}{\sqrt{2}}\x^TH^{(2)}\x)).\label{eq:posexplemmalb}
\end{equation}\label{lemma:posexplemmalb}
\end{lemma}
\begin{proof}
As was the case with the corresponding proof in the previous subsection, the proof is a direct application of Theorem \ref{thm:Slepian2}. Namely, one starts by defining processes $X_i$ and $Y_i$ in the following way
\begin{equation}
Y_i=(\y^{(i)})^T H\x^{(i)} \quad X_i=\frac{1}{\sqrt{2}}(\y^{(i)})^TH^{(1)}\y^{(i)}+\frac{1}{\sqrt{2}}(\x^{(i)})^TH^{(2)}\x^{(i)}.\label{eq:posexplemmaproof1lb}
\end{equation}
Then clearly
\begin{equation}
EY_i^2=EX_i^2=\|\x^{(i)}\|_2^2=1.\label{eq:posexplemmaproof2lb}
\end{equation}
One then further has
\begin{eqnarray}
EY_iY_l & = & (\y^{(i)})^T\y^{(l)}(\x^{(l)})^T\x^{(i)} \nonumber \\
EX_iX_l & = & \frac{1}{2}((\y^{(i)})^T\y^{(l)})^2+\frac{1}{2}((\x^{(l)})^T\x^{(i)})^2.\label{eq:posexplemmaproof3lb}
\end{eqnarray}
And after a small algebraic transformation
\begin{eqnarray}
EX_iX_l-EY_iY_l & = & \frac{1}{2}((\y^{(i)})^T\y^{(l)})^2+\frac{1}{2}((\x^{(l)})^T\x^{(i)})^2-(\y^{(i)})^T\y^{(l)}(\x^{(l)})^T\x^{(i)} \nonumber \\
& = & \frac{1}{2}((\x^{(l)})^T\x^{(i)}-(\y^{(i)})^T\y^{(l)})^2\nonumber \\
& \geq & 0.\label{eq:posexplemmaproof4lb}
\end{eqnarray}
Combining (\ref{eq:posexplemmaproof2lb}) and (\ref{eq:posexplemmaproof4lb}) and using results of Theorem \ref{thm:Slepian2} one then easily obtains (\ref{eq:posexplemmalb}).
\end{proof}

Using results of Lemma \ref{lemma:posexplemmalb} we then have
\begin{multline}
E(\max_{\x\in\{-\frac{1}{\sqrt{n}},\frac{1}{\sqrt{n}}\}^n} \|H\x\|_2) =E(\max_{\x\in\{-\frac{1}{\sqrt{n}},\frac{1}{\sqrt{n}}\}^n,\|\y\|_2=1}(\y^T H\x))\\\geq
E(\max_{\x\in\{-\frac{1}{\sqrt{n}},\frac{1}{\sqrt{n}}\}^n,\|\y\|_2=1}\frac{1}{\sqrt{2}}\y^TH^{(1)}\y+\frac{1}{\sqrt{2}}\x^TH^{(2)}\x)\\
=E(\max_{\|\y\|_2=1}\frac{1}{\sqrt{2}}\y^TH^{(1)}\y)+E(\max_{\x\in\{-\frac{1}{\sqrt{n}},\frac{1}{\sqrt{n}}\}^n}\frac{1}{\sqrt{2}}\x^TH^{(2)}\x),\label{eq:poshopaftlemma2lb}
\end{multline}
and after scaling
\begin{multline}
\frac{E(\max_{\x\in\{-\frac{1}{\sqrt{n}},\frac{1}{\sqrt{n}}\}^n} \|H\x\|_2)}{\sqrt{n}} =\frac{E(\max_{\x\in\{-\frac{1}{\sqrt{n}},\frac{1}{\sqrt{n}}\}^n,\|\y\|_2=1}(\y^T H\x))}{\sqrt{n}}\\\geq
\frac{E(\max_{\|\y\|_2=1}(\y^TH^{(1)}\y))}{\sqrt{2n}}+\frac{E(\max_{\x\in\{-\frac{1}{\sqrt{n}},\frac{1}{\sqrt{n}}\}^n}\x^TH^{(2)}\x))}{\sqrt{2n}}.\label{eq:poshopaftlemma3lb}
\end{multline}
Now, clearly, $\max_{\|\y\|_2=1}(\y^TH^{(1)}\y)$ is the maximum singular value of a Gaussian $m\times m$ matrix $H^{(1)}$. From the theory of large Gaussian random matrices one easily has
\begin{equation}
\lim_{m\rightarrow \infty}\frac{E(\max_{\|\y\|_2=1}(\y^TH^{(1)}\y))}{\sqrt{2m}}=1.\label{eq:singvallimit}
\end{equation}
Moreover, using incredible results of \cite{Parisi80,Tal06,Guerra03} one has
\begin{equation}
\lim_{n\rightarrow \infty}\frac{E(\max_{\x\in\{-\frac{1}{\sqrt{n}},\frac{1}{\sqrt{n}}\}^n}\x^TH^{(2)}\x))}{\sqrt{2n}}=\xi_{SK}\approx 0.763,\label{eq:skmodel}
\end{equation}
where $\xi_{SK}$ is the average ground state energy of the so-called Sherrington-Kirkpatrick (SK) model in the thermodynamic limit. More on the SK model can be found in excellent references \cite{Parisi80,Tal06,Guerra03,SheKir72}. We do mention that the work of \cite{Parisi80,Tal06,Guerra03} indeed settled the thermodynamic behavior of the SK model. However, the characterization of $\xi_{SK}$ in \cite{Parisi80,Tal06,Guerra03} is not explicit and the value we give in (\ref{eq:skmodel}) is a numerical estimate (it is quite likely though, that the estimate we give is a bit conservative; the true value is probably more around $0.7632$). Connecting (\ref{eq:poshopaftlemma3lb}), (\ref{eq:singvallimit}), and (\ref{eq:skmodel}) one then has the following lower-bounding limiting counterpart to (\ref{eq:poshopubexp1})
\begin{equation}
\lim_{n\rightarrow\infty}\frac{E\xi_p}{\sqrt{n}}=\lim_{n\rightarrow\infty}\frac{E(\max_{\x\in\{-\frac{1}{\sqrt{n}},\frac{1}{\sqrt{n}}\}^n} \|H\x\|_2)}{\sqrt{n}} \geq \sqrt{\alpha}+\xi_{SK}\approx\sqrt{\alpha}+0.763.\label{eq:poshopubexplb}
\end{equation}

We now turn to deriving a more general probabilistic result related to $\xi_p$. We will do so through the following lemma (essentially a lower-bounding counterpart to Lemma \ref{lemma:posexplemma}).
\begin{lemma}
Let $H$ be an $m\times n$ matrix with i.i.d. standard normal components. Let $H^{(1)}$ and $H^{(2)}$ be $m\times m$ and $n\times n$ matrices, respectively, with i.i.d. standard normal components. Let $\zeta$ be a scalar. Then
\begin{equation}
P(\max_{\x\in\{-\frac{1}{\sqrt{n}},\frac{1}{\sqrt{n}}\}^n\|\y\|_2=1}(\y^T A\x-\zeta)\geq 0)\geq
P(\max_{\x\in\{-\frac{1}{\sqrt{n}},\frac{1}{\sqrt{n}}\}^n,\|\y\|_2=1}(\frac{1}{\sqrt{2}}\y^TH^{(1)}\y+\frac{1}{\sqrt{2}}\x^TH^{(2)}\x)-\zeta)\geq 0).\label{eq:posproblemmalb}
\end{equation}\label{lemma:posproblemmalb}
\end{lemma}
\begin{proof}
As in the previous subsection, the proof is basically the same as the proof of Lemma \ref{lemma:posexplemmalb}. The only difference is that instead of Theorem \ref{thm:Slepian2} it relies on Theorem \ref{thm:Slepian1}.
\end{proof}

Let $\zeta=\xi_p^{(l)}$. We will first look at the right-hand side of the inequality in (\ref{eq:posproblemmalb}). The following is then the probability of interest
\begin{equation}
P(\max_{\x\in\{-\frac{1}{\sqrt{n}},\frac{1}{\sqrt{n}}\}^n,\|\y\|_2=1}(\frac{1}{\sqrt{2}}\y^TH^{(1)}\y+\frac{1}{\sqrt{2}}\x^TH^{(2)}\x)-\zeta\geq 0).\label{eq:probanal0lb}
\end{equation}
From the theory of large Gaussian random matrices we then have
\begin{equation}
\lim_{n\rightarrow\infty}P(\max_{\|\y\|_2=1}(\frac{1}{\sqrt{2}}\y^TH^{(1)}\y)\geq (1-\epsilon_1^{(m_s)})\sqrt{m})\geq 1.\label{eq:probanal1lb}
\end{equation}
where $\epsilon_1^{(m_s)}$ is an arbitrarily small constant independent of $n$.
The powerful results of \cite{Parisi80,Tal06,Guerra03} also give
\begin{equation}
\lim_{n\rightarrow\infty}P(\max_{\x\in\{-\frac{1}{\sqrt{n}},\frac{1}{\sqrt{n}}\}^n=1}(\frac{1}{\sqrt{2n}}\x^TH^{(2)}\x)\geq (1-\epsilon_1^{(n_{sk})})\xi_{SK})\geq 1,\label{eq:probanal2lb}
\end{equation}
where $\epsilon_1^{(n_{sk})}$ is an arbitrarily small constant independent of $n$. If one then assumes that
\begin{equation}
\xi_p^{(l)}= (1-\epsilon_1^{(m_s)})\sqrt{m}+(1-\epsilon_1^{(n_{sk})}\xi_{SK})\sqrt{n},\label{eq:probanal3lb}
\end{equation}
then a combination of (\ref{eq:probanal0lb}), (\ref{eq:probanal1lb}), and (\ref{eq:probanal2lb}) gives
\begin{multline}
\lim_{n\rightarrow\infty}P(\max_{\x\in\{-\frac{1}{\sqrt{n}},\frac{1}{\sqrt{n}}\}^n,\|\y\|_2=1}(\frac{1}{\sqrt{2}}\y^TH^{(1)}\y+\frac{1}{\sqrt{2}}\x^TH^{(2)}\x)-\zeta\geq 0)\\
=\lim_{n\rightarrow\infty}P(\max_{\x\in\{-\frac{1}{\sqrt{n}},\frac{1}{\sqrt{n}}\}^n,\|\y\|_2=1}(\frac{1}{\sqrt{2}}\y^TH^{(1)}\y+\frac{1}{\sqrt{2}}\x^TH^{(2)}\x)-
((1-\epsilon_1^{(m_s)})\sqrt{m}+(1-\epsilon_1^{(n_{sk})})\xi_{SK}\sqrt{n})\geq 0)\\
\geq \lim_{n\rightarrow\infty}P(\max_{\|\y\|_2=1}(\frac{1}{\sqrt{2}}\y^TH^{(1)}\y)-
(1-\epsilon_1^{(m_s)})\sqrt{m}\geq 0)\\\times \lim_{n\rightarrow\infty}P(\max_{\x\in\{-\frac{1}{\sqrt{n}},\frac{1}{\sqrt{n}}\}^n}(\frac{\x^TH^{(2)}\x}{\sqrt{2}})-(1-\epsilon_1^{(n_{sk})})\xi_{SK}\sqrt{n}\geq 0)\geq 1.\label{eq:probanal4lb}
\end{multline}
Assuming that (\ref{eq:probanal3lb}) holds then a further combination of (\ref{eq:posproblemmalb}) and (\ref{eq:probanal4lb}) gives
\begin{multline}
P(\max_{\x\in\{-\frac{1}{\sqrt{n}},\frac{1}{\sqrt{n}}\}^n}\|H\x\|_2\geq \xi_p^{(l)})=P(\max_{\x\in\{-\frac{1}{\sqrt{n}},\frac{1}{\sqrt{n}}\}^n\|\y\|_2=1}(\y^T A\x)\geq \xi_p^{(l)})\\\geq
\lim_{n\rightarrow\infty}P(\max_{\x\in\{-\frac{1}{\sqrt{n}},\frac{1}{\sqrt{n}}\}^n,\|\y\|_2=1}(\frac{1}{\sqrt{2}}\y^TH^{(1)}\y+\frac{1}{\sqrt{2}}\x^TH^{(2)}\x)-\xi_p^{(l)}\geq 0)\geq 1.\label{eq:probanal5lb}
\end{multline}

We summarize our results from this subsection in the following lemma.

\begin{lemma}
Let $H$ be an $m\times n$ matrix with i.i.d. standard normal components. Let $n$ be large and let $m=\alpha n$, where $\alpha>0$ is a constant independent of $n$. Let $\xi_p$ be as in (\ref{eq:sqrtposham1}). Let $\xi_{SK}$ be the average ground state energy in the thermodynamic limit of the SK model as defined in (\ref{eq:skmodel}). Further, let all $\epsilon$'s be arbitrarily small constants independent of $n$ and let $\xi_p^{(l)}$ be a scalar such that
\begin{equation}
\frac{\xi_p^{(l)}}{\sqrt{n}}= (1-\epsilon_1^{(m_s)})\sqrt{\alpha}+(1-\epsilon_1^{(n_{sk})})\xi_{SK}.\label{eq:condxipuposgenlemmalb}
\end{equation}
Then
\begin{eqnarray}
& & \lim_{n\rightarrow\infty}P(\max_{\x\in\{-\frac{1}{\sqrt{n}},\frac{1}{\sqrt{n}}\}^n}(\|H\x\|_2)\geq \xi_p^{(l)})\geq 1\nonumber \\
& \Leftrightarrow & \lim_{n\rightarrow\infty}P(\xi_p\geq \xi_p^{(l)})\geq 1 \nonumber \\
& \Leftrightarrow & \lim_{n\rightarrow\infty}P(\xi_p^2\geq (\xi_p^{(l)})^2)\geq 1, \label{eq:posgenproblemmalb}
\end{eqnarray}
and
\begin{equation}
\lim_{n\rightarrow\infty}\frac{E\xi_p}{\sqrt{n}}=\lim_{n\rightarrow\infty}\frac{E(\max_{\x\in\{-\frac{1}{\sqrt{n}},\frac{1}{\sqrt{n}}\}^n} \|H\x\|_2)}{\sqrt{n}} \geq \sqrt{\alpha}+\xi_{SK}\approx \sqrt{\alpha}+0.763.\label{eq:posgenexplemmalb}
\end{equation}
\label{lemma:posgenlemmalb}
\end{lemma}
\begin{proof}
The proof follows from the above discussion, (\ref{eq:poshopubexplb}), and (\ref{eq:probanal5lb}).
\end{proof}

A combination of results obtained in Lemmas (\ref{lemma:posgenlemma}) and (\ref{lemma:posgenlemmalb}) then gives
\begin{equation}
\sqrt{\alpha}+0.763\leq \approx \sqrt{\alpha}+\xi_{SK} \leq \lim_{n\rightarrow\infty}\frac{E\xi_p}{\sqrt{n}}=\lim_{n\rightarrow\infty}\frac{E(\max_{\x\in\{-\frac{1}{\sqrt{n}},\frac{1}{\sqrt{n}}\}^n} \|H\x\|_2)}{\sqrt{n}} \leq \sqrt{\alpha}+\sqrt{\frac{2}{\pi}}\approx \sqrt{\alpha}+0.798.\label{eq:posublb}
\end{equation}
Although we don't go into further analytical considerations as to what happens with the above bounds as $\alpha$ changes, we do mention that as $\alpha\rightarrow 0$ the upper bound is expected to be close to the true value. On the other hand, as $\alpha\rightarrow \infty$ the lower bound is expected to be close to the true value (for more in this direction see, e.g. \cite{JYZhao11}). A massive set of numerical experiments that we performed (and that we will report on in a forthcoming paper) seems to indicate that this indeed is a trend. In other words, as $\alpha$ grows from zero to $\infty$ the true value of $\lim_{n\rightarrow\infty}\frac{E\xi_p}{\sqrt{n}}$ seems to slowly transition from the most left to the most right quantity given in (\ref{eq:posublb}).

\section{Negative Hopfield form}
\label{sec:neghop}

In this section we will look at the following optimization problem (which again clearly is the key component in estimating the corresponding ground state energy of what we call the negative Hopfield model in the thermodynamic limit)
\begin{equation}
\min_{\x\in\{-\frac{1}{\sqrt{n}},\frac{1}{\sqrt{n}}\}^n}\|H\x\|_2^2.\label{eq:negham1}
\end{equation}
Similarly to what was the case when we studied the positive form in the previous section, for a deterministic (given fixed) $H$ the above problem is of course known to be NP-hard. Of course, this is same as was the case for (\ref{eq:posham1}) as it again essentially falls under the class of binary quadratic optimization problems. Consequently, we will again adopt a strategy similar to the one that we considered when studied the positive form in the previous section. Namely, instead of looking at the problem in (\ref{eq:negham1}) in a deterministic way i.e. in a way that assumes that matrix $H$ is deterministic, we will look at it in a statistical scenario. Also as in previous section, we will assume that the elements of matrix $H$ are i.i.d. standard normals. We will then call the form (\ref{eq:posham1}) with Gaussian $H$, the Gaussian negative Hopfield form. On the other hand, we will call the form (\ref{eq:negham1}) with Bernoulli $H$, the Bernoulli negative Hopfield form. In the remainder of this section we will look at possible ways to estimate the optimal value of the optimization problem in (\ref{eq:negham1}). In fact we will introduce a strategy similar the one presented in the previous section to create a lower-bound on the optimal value of (\ref{eq:negham1}).

\subsection{Lower-bounding ground state energy of the negative Hopfield form}
\label{sec:neghoplb}

In this section we will look at problem from (\ref{eq:negham1}).In fact, to be a bit more precise, as in the previous section, in order to make the exposition as simple as possible, we will look at its a slight variant given below
\begin{equation}
\xi_n=\min_{\x\in\{-\frac{1}{\sqrt{n}},\frac{1}{\sqrt{n}}\}^n}\|H\x\|_2.\label{eq:sqrtnegham1}
\end{equation}
As mentioned above, we will assume that the elements of $H$ are i.i.d. standard normal random variables. Now, to create a lower-bounding strategy for the negative Hopfield form we will rely on Theorems \ref{thm:Gordonmesh1} and Theorem \ref{thm:Gordonmesh2}. We start by reformulating the problem in (\ref{eq:sqrtnegham1}) in the following way
\begin{equation}
\xi_n=\min_{\x\in\{-\frac{1}{\sqrt{n}},\frac{1}{\sqrt{n}}\}^n}\max_{\|\y\|_2=1}\y^TH\x.\label{eq:sqrtnegham2}
\end{equation}
As in the previous section, we will first focus on the expected value of $\xi_n$ and then on its more general probabilistic properties. The following is then a direct application of Theorem \ref{thm:Gordonmesh2}.
\begin{lemma}
Let $H$ be an $m\times n$ matrix with i.i.d. standard normal components. Let $\g$ and $\h$ be $m\times 1$ and $n\times 1$ vectors, respectively, with i.i.d. standard normal components. Also, let $g$ be a standard normal random variable. Then
\begin{equation}
E(\min_{\x\in\{-\frac{1}{\sqrt{n}},\frac{1}{\sqrt{n}}\}^n}\max_{\|\y\|_2=1}(\y^T H\x +\|\x\|_2 g))\geq E(\min_{\x\in\{-\frac{1}{\sqrt{n}},\frac{1}{\sqrt{n}}\}^n}\max_{\|\y\|_2=1}(\|\x\|_2\g^T\y+\h^T\x)).\label{eq:negexplemma}
\end{equation}\label{lemma:negexplemma}
\end{lemma}
\begin{proof}
As mentioned above, the proof is a standard/direct application of Theorem \ref{thm:Gordonmesh2}. We will sketch it for completeness. Namely, one starts by defining processes $X_i$ and $Y_i$ in the following way
\begin{equation}
Y_{ij}=(\y^{(j)})^T H\x^{(i)} +\|\x^{(i)}\|_2 g\quad X_{ij}=\|\x^{(i)}\|_2\g^T\y^{(j)}+\h^T\x^{(i)}.\label{eq:negexplemmaproof1}
\end{equation}
Then clearly
\begin{equation}
EY_{ij}^2=EX_{ij}^2=2\|\x^{(i)}\|_2^2=2.\label{eq:negexplemmaproof2}
\end{equation}
One then further has
\begin{eqnarray}
EY_{ij}Y_{ik} & = & (\x^{(i)})^T\x^{(i)}(\y^{(k)})^T\y^{(j)}+\|\x^{(i)}\|_2\|\x^{(i)}\|_2 \nonumber \\
EX_{ij}X_{ik} & = & \|\x^{(i)}\|_2\|\x^{(i)}\|_2(\y^{(k)})^T\y^{(j)}+(\x^{(i)})^T\x^{(i)},\label{eq:negexplemmaproof3}
\end{eqnarray}
and clearly
\begin{equation}
EX_{ij}X_{ik}=EY_{ij}Y_{ik}.\label{eq:negexplemmaproof31}
\end{equation}
Moreover,
\begin{eqnarray}
EY_{ij}Y_{lk} & = & (\y^{(j)})^T\y^{(k)}(\x^{(i)})^T\x^{(l)}+\|\x^{(i)}\|_2\|\x^{(l)}\|_2 \nonumber \\
EX_{ij}X_{lk} & = & (\y^{(j)})^T\y^{(k)}\|\x^{(i)}\|_2\|\x^{(l)}\|_2+(\x^{(i)})^T\x^{(l)}.\label{eq:negexplemmaproof32}
\end{eqnarray}
And after a small algebraic transformation
\begin{eqnarray}
EY_{ij}Y_{lk}-EX_{ij}X_{lk} & = & \|\x^{(i)}\|_2\|\x^{(l)}\|_2(1-(\y^{(j)})^T\y^{(k)})-(\x^{(i)})^T\x^{(l)}(1-(\y^{(j)})^T\y^{(k)}) \nonumber \\
& = & (\|\x^{(i)}\|_2\|\x^{(l)}\|_2-(\x^{(i)})^T\x^{(l)})(1-(\y^{(j)})^T\y^{(k)})\nonumber \\
& \geq & 0.\label{eq:negexplemmaproof4}
\end{eqnarray}
Combining (\ref{eq:negexplemmaproof2}), (\ref{eq:negexplemmaproof31}), and (\ref{eq:negexplemmaproof4}) and using results of Theorem \ref{thm:Gordonmesh2} one then easily obtains (\ref{eq:negexplemma}).
\end{proof}

Using results of Lemma \ref{lemma:negexplemma} we then have
\begin{multline}
E(\min_{\x\in\{-\frac{1}{\sqrt{n}},\frac{1}{\sqrt{n}}\}^n} \|H\x\|_2) =E(\min_{\x\in\{-\frac{1}{\sqrt{n}},\frac{1}{\sqrt{n}}\}^n}\max_{\|\y\|_2=1}(\y^T H\x +\|\x\|_2g))\\ \geq E(\min_{\x\in\{-\frac{1}{\sqrt{n}},\frac{1}{\sqrt{n}}\}^n}\max_{\|\y\|_2=1}(\|\x\|_2\g^T\y+\h^T\x))=E\|\x\|_2\|\g\|_2-E\sum_{i=1}^{n}|\h_i|\geq \sqrt{n}(\sqrt{m}-\frac{1}{4\sqrt{m}})-\sqrt{\frac{2}{\pi}}n.\label{eq:neghopaftlemma2}
\end{multline}
Connecting beginning and end of (\ref{eq:neghopaftlemma2}) we finally have a lower bound on $E\xi_n$ from (\ref{eq:sqrtnegham1}), i.e.
\begin{equation}
E\xi_n=E(\min_{\x\in\{-\frac{1}{\sqrt{n}},\frac{1}{\sqrt{n}}\}^n} \|H\x\|_2) \geq (\sqrt{m}-\frac{1}{4\sqrt{m}})-\sqrt{\frac{2}{\pi}}\sqrt{n}=\sqrt{n}(\sqrt{\alpha}-\frac{1}{4\sqrt{mn}}-\sqrt{\frac{2}{\pi}}),\label{eq:neghopubexp}
\end{equation}
or in a scaled (possibly) more convenient form
\begin{equation}
\frac{E\xi_n}{\sqrt{n}}=\frac{E(\max_{\x\in\{-\frac{1}{\sqrt{n}},\frac{1}{\sqrt{n}}\}^n} \|H\x\|_2)}{\sqrt{n}} \geq \sqrt{\alpha}-\frac{1}{4\sqrt{mn}}-\sqrt{\frac{2}{\pi}}.\label{eq:neghopubexp1}
\end{equation}
Of course, the above result will be useful as long as the most right quantity is positive.

Following what was done in the previous section we will now turn to deriving a more general probabilistic result related to $\xi_n$ (all the comments related to these type of results that we have made in the previous section apply here as well). We will do so through the following lemma.
\begin{lemma}
Let $H$ be an $m\times n$ matrix with i.i.d. standard normal components. Let $\g$ and $\h$ be $m\times 1$ and $n\times 1$ vectors, respectively, with i.i.d. standard normal components. Also, let $g$ be a standard normal random variable and let $\zeta_{\x}$ be a function of $\x$. Then
\begin{equation}
P(\min_{\x\in\{-\frac{1}{\sqrt{n}},\frac{1}{\sqrt{n}}\}^n}\max_{\|\y\|_2=1}(\y^T A\x+\|\x\|_2g-\zeta_{\x})\geq 0)\geq
P(\min_{\x\in\{-\frac{1}{\sqrt{n}},\frac{1}{\sqrt{n}}\}^n}\max_{\|\y\|_2=1}(\|\x\|_2\g^T\y+\h^T\x-\zeta_{\x})\geq 0).\label{eq:negproblemma}
\end{equation}\label{lemma:negproblemma}
\end{lemma}
\begin{proof}
The proof is basically same as the proof of Lemma \ref{lemma:negexplemma}. The only difference is that instead of Theorem \ref{thm:Gordonmesh2} it relies on Theorem \ref{thm:Gordonmesh1}.
\end{proof}

Let $\zeta_{\x}=\epsilon_{5}^{(g)}\sqrt{n}\|\x\|_2+\xi_n^{(l)}$ with $\epsilon_{5}^{(g)}>0$ being an arbitrarily small constant independent of $n$. We will first look at the right-hand side of the inequality in (\ref{eq:negproblemma}). The following is then the probability of interest
\begin{equation}
P(\min_{\x\in\{-\frac{1}{\sqrt{n}},\frac{1}{\sqrt{n}}\}^n}\max_{\|\y\|_2=1}(\|\x\|_2\g^T\y+\h^T\x-\epsilon_{5}^{(g)}\sqrt{n}\|\x\|_2)\geq \xi_n^{(l)}).\label{eq:negprobanal0}
\end{equation}
After solving the minimization over $\x$ and the maximization over $\y$ one obtains
\begin{equation}
\hspace{-.3in}P(\min_{\x\in\{-\frac{1}{\sqrt{n}},\frac{1}{\sqrt{n}}\}^n}\max_{\|\y\|_2=1}(\|\x\|_2\g^T\y+\h^T\x-\epsilon_{5}^{(g)}\sqrt{n}\|\x\|_2)\geq \xi_n^{(l)})=P(\|\g\|_2-\sum_{i=1}^{n}|\h_i|/\sqrt{n}-\epsilon_{5}^{(g)}\sqrt{n}\geq \xi_n^{(l)}).\label{eq:negprobanal1}
\end{equation}
We recall that as earlier, since $\g$ is a vector of $m$ i.i.d. standard normal variables it is rather trivial that $P(\|\g\|_2>(1-\epsilon_{1}^{(m)})\sqrt{m})\geq 1-e^{-\epsilon_{2}^{(m)} m}$ where $\epsilon_{1}^{(m)}>0$ is an arbitrarily small constant and $\epsilon_{2}^{(m)}$ is a constant dependent on $\epsilon_{1}^{(m)}$ but independent of $n$. Along the same lines, since $\h$ is a vector of $n$ i.i.d. standard normal variables it is rather trivial that $P(\sum_{i=1}^{n}|\h_i|<(1+\epsilon_{1}^{(n)})n\sqrt{\frac{2}{\pi}})\geq 1-e^{-\epsilon_{2}^{(n)} n}$ where $\epsilon_{1}^{(n)}>0$ is an arbitrarily small constant and $\epsilon_{2}^{(n)}$ is a constant dependent on $\epsilon_{1}^{(n)}$ but independent of $n$. Then from (\ref{eq:negprobanal1}) one obtains
\begin{multline}
P(\min_{\x\in\{-\frac{1}{\sqrt{n}},\frac{1}{\sqrt{n}}\}^n}\max_{\|\y\|_2=1}(\|\x\|_2\g^T\y+\h^T\x-\epsilon_{5}^{(g)}\sqrt{n}\|\x\|_2)\geq \xi_n^{(l)})\\\geq
(1-e^{-\epsilon_{2}^{(m)} m})(1-e^{-\epsilon_{2}^{(n)} n})
P((1-\epsilon_{1}^{(m)})\sqrt{m}-(1+\epsilon_{1}^{(n)})\sqrt{n}\sqrt{\frac{2}{\pi}}-\epsilon_{5}^{(g)}\sqrt{n}\geq \xi_n^{(l)}).
\label{eq:negprobanal2}
\end{multline}
If
\begin{eqnarray}
& & (1-\epsilon_{1}^{(m)})\sqrt{m}-(1+\epsilon_{1}^{(n)})\sqrt{n}\sqrt{\frac{2}{\pi}}-\epsilon_{5}^{(g)}\sqrt{n}>\xi_n^{(l)}\nonumber \\
& \Leftrightarrow & (1-\epsilon_{1}^{(m)})\sqrt{\alpha}-(1+\epsilon_{1}^{(n)})\sqrt{\frac{2}{\pi}}-\epsilon_{5}^{(g)}>\frac{\xi_n^{(l)}}{\sqrt{n}},\label{eq:negcondxipu}
\end{eqnarray}
one then has from (\ref{eq:negprobanal2})
\begin{equation}
\lim_{n\rightarrow\infty}P(\min_{\x\in\{-\frac{1}{\sqrt{n}},\frac{1}{\sqrt{n}}\}^n}\max_{\|\y\|_2=1}(\|\x\|_2\g^T\y+\h^T\x-\epsilon_{5}^{(g)}\sqrt{n}\|\x\|_2)\geq \xi_n^{(l)})\geq 1.\label{eq:negprobanal3}
\end{equation}

We will now look at the left-hand side of the inequality in (\ref{eq:negproblemma}). The following is then the probability of interest
\begin{equation}
P(\min_{\x\in\{-\frac{1}{\sqrt{n}},\frac{1}{\sqrt{n}}\}^n}\max_{\|\y\|_2=1}(\y^TH\x+\|\x\|_2g-\epsilon_{5}^{(g)}\sqrt{n}\|\x\|_2-\xi_n^{(l)})\geq 0).\label{eq:leftnegprobanal0}
\end{equation}
Since $P(g\geq\epsilon_{5}^{(g)}\sqrt{n})<e^{-\epsilon_{6}^{(g)} n}$ (where $\epsilon_{6}^{(g)}$ is, as all other $\epsilon$'s in this paper are, independent of $n$) from (\ref{eq:leftnegprobanal0}) we have
\begin{multline}
P(\min_{\x\in\{-\frac{1}{\sqrt{n}},\frac{1}{\sqrt{n}}\}^n}\max_{\|\y\|_2=1}(\y^TH\x+\|\x\|_2g-\epsilon_{5}^{(g)}\sqrt{n}\|\x\|_2-\xi_n^{(l)})\geq 0)
\\\leq P(\min_{\x\in\{-\frac{1}{\sqrt{n}},\frac{1}{\sqrt{n}}\}^n}\max_{\|\y\|_2=1}(\y^TH\x-\xi_n^{(l)})\geq 0)+e^{-\epsilon_{6}^{(g)} n}.\label{eq:leftnegprobanal1}
\end{multline}
When $n$ is large from (\ref{eq:leftnegprobanal1}) we then have
\begin{multline}
\hspace{-.65in}\lim_{n\rightarrow \infty}P(\min_{\x\in\{-\frac{1}{\sqrt{n}},\frac{1}{\sqrt{n}}\}^n}\max_{\|\y\|_2=1}(\y^TH\x+\|\x\|_2g-\epsilon_{5}^{(g)}\sqrt{n}\|\x\|_2-\xi_n^{(l)})\geq 0)
\leq  \lim_{n\rightarrow\infty}P(\min_{\x\in\{-\frac{1}{\sqrt{n}},\frac{1}{\sqrt{n}}\}^n}\max_{\|\y\|_2=1}(\y^TH\x-\xi_n^{(l)})\geq 0)\\
 =  \lim_{n\rightarrow\infty}P(\min_{\x\in\{-\frac{1}{\sqrt{n}},\frac{1}{\sqrt{n}}\}^n}\max_{\|\y\|_2=1}(\y^TH\x)\geq \xi_n^{(l)})
 =  \lim_{n\rightarrow\infty}P(\min_{\x\in\{-\frac{1}{\sqrt{n}},\frac{1}{\sqrt{n}}\}^n}(\|H\x\|_2)\geq \xi_n^{(l)}).\label{eq:leftnegprobanal2}
\end{multline}
Assuming that (\ref{eq:negcondxipu}) holds, then a combination of (\ref{eq:negproblemma}), (\ref{eq:negprobanal3}), and (\ref{eq:leftnegprobanal2}) gives
\begin{equation}
 \lim_{n\rightarrow\infty}P(\min_{\x\in\{-\frac{1}{\sqrt{n}},\frac{1}{\sqrt{n}}\}^n}(\|H\x\|_2)\geq \xi_n^{(l)})\geq \lim_{n\rightarrow\infty}P(\min_{\x\in\{-\frac{1}{\sqrt{n}},\frac{1}{\sqrt{n}}\}^n}\max_{\|\y\|_2=1}(\|\x\|_2\g^T\y+\h^T\x-\epsilon_{5}^{(g)}\sqrt{n}\|\x\|_2)\geq \xi_n^{(l)})\geq 1.\label{eq:leftnegprobanal3}
\end{equation}

We summarize our results from this subsection in the following lemma.

\begin{lemma}
Let $H$ be an $m\times n$ matrix with i.i.d. standard normal components. Let $n$ be large and let $m=\alpha n$, where $\alpha>0$ is a constant independent of $n$. Let $\xi_n$ be as in (\ref{eq:sqrtnegham1}). Let all $\epsilon$'s be arbitrarily small constants independent of $n$ and let $\xi_n^{(l)}$ be a scalar such that
\begin{equation}
(1-\epsilon_{1}^{(m)})\sqrt{\alpha}-(1+\epsilon_{1}^{(n)})\sqrt{\frac{2}{\pi}}-\epsilon_{5}^{(g)}>\frac{\xi_n^{(l)}}{\sqrt{n}}.\label{eq:negcondxipuneggenlemma}
\end{equation}
Then
\begin{eqnarray}
& & \lim_{n\rightarrow\infty}P(\min_{\x\in\{-\frac{1}{\sqrt{n}},\frac{1}{\sqrt{n}}\}^n}(\|H\x\|_2)\geq \xi_n^{(l)})\geq 1\nonumber \\
& \Leftrightarrow & \lim_{n\rightarrow\infty}P(\xi_n\geq \xi_n^{(l)})\geq 1 \nonumber \\
& \Leftrightarrow & \lim_{n\rightarrow\infty}P(\xi_n^2\geq (\xi_n^{(l)})^2)\geq 1, \label{eq:neggenproblemma}
\end{eqnarray}
and
\begin{equation}
\frac{E\xi_n}{\sqrt{n}}=\frac{E(\max_{\x\in\{-\frac{1}{\sqrt{n}},\frac{1}{\sqrt{n}}\}^n} \|H\x\|_2)}{\sqrt{n}} \geq \sqrt{\alpha}-\frac{1}{4\sqrt{mn}}-\sqrt{\frac{2}{\pi}}.\label{eq:neggenexplemma}
\end{equation}
\label{lemma:neggenlemma}
\end{lemma}
\begin{proof}
The proof follows from the above discussion, (\ref{eq:neghopubexp1}), and (\ref{eq:leftnegprobanal3}).
\end{proof}

\section{Algorithmic aspects of Hopfield forms}
\label{sec:alghop}

In this section we look at a couple of simple algorithms that can be used to approximately solve optimization problems we studied in the previous sections. The algorithms are clearly not the best possible but are fairly simple. Given their simple structure it will turn out to be possible to provide an analytical characterization of the optimal values that they achieve. In return these values would automatically become bounds on the true optimal values. These bounds won't be as good as those we presented in the previous sections but will in a way be their algorithmic complements. As earlier in the paper, we will start with the positive Hopfield form and then we will present the corresponding results for the negative Hopfield form.

\subsection{Simple approximate algorithms for the positive Hopfield forms}
\label{sec:alghoppos}

We recall that our goal it this subsection will be to present algorithms that provide an approximate solution to (\ref{eq:posham1}) (or alternatively (\ref{eq:sqrtposham1})). Before, proceeding further we recall that in the previous couple of sections it was a bit easier to focus on (\ref{eq:sqrtposham1}) instead of focusing on (\ref{eq:posham1}). In this section though, it will be the other way around, i.e. we will focus on the original problem (\ref{eq:posham1}) which we restate below
\begin{equation}
\xi_p^2=\max_{\x\in\{-\frac{1}{\sqrt{n}},\frac{1}{\sqrt{n}}\}^n}\|H\x\|_2^2.\label{eq:posham1alg}
\end{equation}
In this section we will present two simple approximate algorithms that can be used to solve approximately (\ref{eq:posham1alg}). We will first present an iterative algorithm that fixes components of $\x$ one at the time and then an algorithm based on the properties of eigenvalues and eigenvectors of Gaussian random matrices.

\subsubsection{An iterative approximate algorithm for the positive Hopfield forms}
\label{sec:alghopposit}

In this section we present an iterative algorithm that approximately solves (\ref{eq:posham1alg}). The algorithm is very simple and probably well known. However, we are not aware of any analytical results related to its quality of performance when applied in a statistical scenario considered in this paper. The analysis is actually fairly simple and we think it would be useful to have such a result recorded. Also, since it will be a bit easier to present and follow the exposition we will until the end of this subsection assume that everything is rescaled so that $\x_i\in\{-1,1\}$. Now, going back to the algorithm - as we just stated the algorithm is fairly simple: it starts by setting $\x_1=1$ and then fixing $\x_2$ so that $\|H_{:,1:2}\x_{1:2}\|_2^2$ is maximized ($H_{:,1:2}$ stands for the first two columns of $H$ and $\x_{1:2}$ stands for the first two components of $\x$). After $\x_2$ is fixed the algorithm then proceeds by fixing $\x_3$ so that $\|H_{:,1:3}\x_{1:3}\|_2^2$ is maximized ($H_{:,1:3}$ stands for the first three columns of $H$ and $\x_{1:3}$ stands for the first three components of $\x$) and so on until one fixes all components of $\x$.

To analyze the algorithm we will set $\hat{\x}_1=1$, $r_1=\|H_{:,1}\|_2^2$, and for any $2\leq k\leq n$
\begin{eqnarray}
\hat{\x}_k & = & \mbox{argmax}_{\x_k\in\{-1,1\}}\|H_{:,1:k}\begin{bmatrix} \hat{\x}_{1:k-1}\\ \x_k\end{bmatrix} \|_2^2\nonumber \\
r_k & = & \max_{\x_k\in\{-1,1\}}\|H_{:,1:k}\begin{bmatrix} \hat{\x}_{1:k-1}\\ \x_k\end{bmatrix} \|_2^2=\|H_{:,1:k}\hat{\x}_{1:k}\|_2^2.\label{eq:defrposit}
\end{eqnarray}
Our goal will be to compute $Er_n$. We will do so in a recursive fashion. To that end we will start with $Er_2$
\begin{multline}
Er_2=\max_{\x_2\in\{-1,1\}}\|H_{:,1:2}\begin{bmatrix} \hat{\x}_{1}\\ \x_2\end{bmatrix} \|_2^2
=\max_{\x_2\in\{-1,1\}}\|H_{:,1:2}\begin{bmatrix} 1\\ \x_2\end{bmatrix} \|_2^2\\=E\|H_{:,1}\|_2^2+2E(\max_{\x_2\in\{-1,1\}}\x_2(H_{:,2}^TH_{:,1}))+E\|H_{:,2}\|_2^2
=Er_1+2\sqrt{\frac{2}{\pi}}E\sqrt{r_1}+m.\label{eq:rposit1}
\end{multline}
One can then apply a similar strategy to obtain for a general $2\leq k\leq n$
\begin{multline}
Er_k=\max_{\x_k\in\{-1,1\}}\|H_{:,1:k}\begin{bmatrix} \hat{\x}_{1:k-1}\\ \x_k\end{bmatrix} \|_2^2
=\max_{\x_k\in\{-1,1\}}\|\begin{bmatrix}H_{:,1:k-1} & H_{:,k} \end{bmatrix}\begin{bmatrix} \hat{\x}_{1:k-1}\\ \x_k\end{bmatrix} \|_2^2\\=E\|H_{:,1:k-1}\hat{\x}_{1:k-1}\|_2^2+2E(\max_{\x_k\in\{-1,1\}}\x_k(H_{:,k}^TH_{:,1:k-1}\hat{\x}_{1:k-1}))+E\|H_{:,k}\|_2^2
=Er_{k-1}+2\sqrt{\frac{2}{\pi}}E\sqrt{r_{k-1}}+m.\label{eq:rposit2}
\end{multline}
To make the exposition easier we will assume that $n$ is large and switch to the limiting behavior of $Er$'s. Assuming concentration of $r_k$'s (for $k$ proportional to $n$) around their mean values gives $\lim_{n\rightarrow\infty} \frac{E\sqrt{r_k}}{n}=\lim_{n\rightarrow\infty} \frac{\sqrt{Er_k}}{n}$. One then based on (\ref{eq:defrposit}), (\ref{eq:rposit1}), and (\ref{eq:rposit2}) can establish the following recursion for finding $Er_n$
\begin{equation}
\phi_k=\phi_{k-1}+2\sqrt{\frac{2}{\pi}}\sqrt{\phi_{k-1}}+m,\label{eq:rposit3}
\end{equation}
with $\phi_1=m$ and $\lim_{n\rightarrow\infty}\frac{Er_n}{n}=\lim_{n\rightarrow\infty}\frac{\phi_n}{n}$. Computing the last limit can then be done to a fairly high precision for any different $m$. We do mention, for example that for $m=n$ (i.e. $\alpha=1$) one has
\begin{equation}
\lim_{n\rightarrow\infty}\frac{Er_n}{n^2}=\lim_{n\rightarrow\infty}\frac{\phi_n}{n^2}\approx 2.5259.\label{eq:numpositlb}
\end{equation}
One can also compare this result to the results of the previous section to get
\begin{equation}
\lim_{n\rightarrow\infty}\frac{E\xi_p}{\sqrt{n}}\geq\lim_{n\rightarrow\infty}\frac{E\sqrt{r_n}}{n}=\lim_{n\rightarrow\infty}\frac{\sqrt{\phi_n}}{n}\approx \sqrt{2.5259}\approx 1.5893.\label{eq:numpositlb1}
\end{equation}
This is a bit worse than $1.763$ bound one would get in Subsection \ref{sec:poshoplb} when $\alpha=1$ (i.e. $m=n$)). However, the bound in (\ref{eq:numpositlb1}) is algorithmic, i.e. there is an algorithm (in fact a very simple one with a quadratic complexity) that achieves it, whereas the bound from Subsection \ref{sec:poshoplb} is purely theoretical and is given without any polynomial algorithm that achieves it.

\subsubsection{A dominating eigenvector algorithm for the positive Hopfield forms}
\label{sec:alghopposeig}

In this section we present another simple algorithm that approximately solves (\ref{eq:posham1alg}). This algorithm is also probably well known, but we think that it would a good idea to collect at one place the technical results related to the objective value one can get through it. In that way it will be easier to know how far away from the optimal its performance is.

As the name suggests the algorithm operates on eigenvectors of $H$. The idea is to decompose $H^TH$ through the eigen-decomposition in the following way
\begin{equation}
H^TH=Q\Lambda Q^T,\label{eq:eigdec}
\end{equation}
where obviously $Q$ is an $n\times n$ matrix such that $Q^TQ=I$ and $\Lambda$ is a diagonal matrix of all eigenvalues of matrix $H^TH$. Now, without a loss of generality, we will assume that the elements of the diagonal matrix $\Lambda$ (essentially the eigenvalues of $H^TH$) are sorted in the decreasing order, i.e. $\Lambda_{1,1}\geq \Lambda_{2,2}\geq \dots\geq \Lambda_{n,n}$. The algorithm then works in the following simple way: take $\x$ as the signs of components of vector $Q_{:,1}$, i.e.
\begin{equation}
\hat{\x}^{(eig)}=\mbox{sign}(Q_{:,1}).\label{eq:eigoptx}
\end{equation}
Let
\begin{equation}
r^{(eig)}=\|H\hat{\x}^{(eig)}\|_2^2=(\mbox{sign}(Q_{:,1}))^TQ\Lambda Q^T\mbox{sign}(Q_{:,1}).\label{eq:eigoptr}
\end{equation}
One then further has
\begin{equation}
r^{(eig)}=(\mbox{sign}(Q_{:,1}))^TQ\Lambda Q^T\mbox{sign}(Q_{:,1})\geq \Lambda_{1,1}(\sum_{i=1}^{n}|Q_{i,1}|)^2.\label{eq:eigoptr1}
\end{equation}
Using the theory of random Gaussian matrices one then has that all quantities of interest concentrate and
\begin{equation}
\lim_{n\rightarrow\infty}\frac{E\Lambda_{1,1}}{n}=(\sqrt{\alpha}+1)^2.\label{eq:eigoptr2}
\end{equation}
Furthermore, one can think of all components of $Q_{:,1}$ as being standard normal scaled by the the norm-2 of the vector they comprise. Since everything concentrates when $n$ is large one then has
\begin{equation}
\lim_{n\rightarrow\infty}E(\frac{\sum_{i=1}^{n}|Q_{i,1}|}{\sqrt{n}})^2=(\sqrt{\frac{2}{\pi}})^2=\frac{2}{\pi}.\label{eq:eigoptr3}
\end{equation}
A combination of (\ref{eq:eigoptr1}), (\ref{eq:eigoptr2}), and (\ref{eq:eigoptr3}) then gives
\begin{equation}
\lim_{n\rightarrow\infty}\frac{Er^{(eig)}}{n^2}\geq \lim_{n\rightarrow\infty}\frac{\Lambda_{1,1}(\sum_{i=1}^{n}|Q_{i,1}|)^2}{n^2}=(\sqrt{\alpha}+1)^2\frac{2}{\pi}.\label{eq:eigoptr4}
\end{equation}
One can also compare this result to the results of the previous section. For example, let $\alpha=1$ and
\begin{equation}
\lim_{n\rightarrow\infty}\frac{E\xi_p}{n}\geq\lim_{n\rightarrow\infty}\frac{E\sqrt{r^{(eig)}}}{n}\geq \lim_{n\rightarrow\infty}\sqrt{\frac{8}{\pi}}\approx \sqrt{2.5465}\approx 1.5958.\label{eq:numposeiglb1}
\end{equation}
This is again somewhat worse than $1.763$ bound one would get in Subsection \ref{sec:poshoplb} when $\alpha=1$ (i.e. $m=n$) but a bit better than what one can get through the mechanism of the previous subsection and ultimately (\ref{eq:numpositlb1}). However, the bound in (\ref{eq:numposeiglb1}) is again algorithmic. The corresponding algorithm though is a bit more complex than the one from the previous subsection since it involves performing the eigen-decomposition of $H^TH$. However, we should mention that the value given in (\ref{eq:numposeiglb1}) is substantially lower than what the algorithm will indeed give in practice. The reason is of course the cross-correlation of components of different eigenvectors and the fact that the cross products between $\hat{\x}^{(eig)}$ and vectors $Q_{i,:},2\leq i\leq n$, coupled with corresponding eigenvalues will also contribute to the true value of $r^{(eig)}$. To obtain the exact value of $\lim_{n\rightarrow\infty}\frac{Er^{(eig)}}{n^2}$ one would have to account for this as well. This is not so easy and we do not pursue it further. However, practically speaking we do mention, that roughly one can expect that $\lim_{n\rightarrow\infty}\frac{Er^{(eig)}}{n^2}\approx 2.9$ or stated differently $\lim_{n\rightarrow\infty}\frac{E\sqrt{r^{(eig)}}}{n}\approx 1.7$. On the other hand, to be completely fair to the algorithm given in the previous subsection, we should mention that its various adaptations are possible as well. For example, among the simplest ones would be to also keep sorting the columns of $H$ and in each step instead of choosing the first next column choose the column with the largest norm-2. Evaluating the performance of such an algorithm precisely is again not super easy. We do mention from practical experience that it provides a similar objective value as does the eigenvector mechanism presented in this subsection.

\subsection{A simple approximate algorithm for the negative Hopfield forms}
\label{sec:alghopneg}

We recall that our goal it this subsection will be to present algorithms that provide an approximate solution to (\ref{eq:negham1}) (or alternatively (\ref{eq:sqrtnegham1})). Before proceeding further we note that in Section \ref{sec:neghoplb} it was be a bit easier to focus on (\ref{eq:sqrtnegham1}) instead of focusing on (\ref{eq:negham1}). In this section though, it will be the other way around, i.e. we will focus on the original problem (\ref{eq:negham1}) which we restate below
\begin{equation}
\xi_n^2=\min_{\x\in\{-\frac{1}{\sqrt{n}},\frac{1}{\sqrt{n}}\}^n}\|H\x\|_2^2.\label{eq:negham1alg}
\end{equation}
Below we will present a simple approximate algorithm that can be used to solve approximately (\ref{eq:negham1alg}). The algorithm will be a counterpart for the negative form to the iterative algorithm given in Section \ref{sec:alghopposit} for the positive Hopfield form.

\subsubsection{An iterative approximate algorithm for the negative Hopfield forms}
\label{sec:alghopnegit}

As mentioned above, in this section we present a counterpart to the iterative algorithm given in Subsection \ref{sec:alghopposit}. Clearly, the algorithm that we will present here approximately solves (\ref{eq:negham1alg}). In fact as when we looked at the positive form we will again assume that everything is scaled so that $\x_i\in\{-1,1\}$. In fact, the algorithm is almost the same as the algorithm from Subsection \ref{sec:alghopposit}: it starts by setting $\x_1=1$ and then fixing $\x_2$ so that $\|H_{:,1:2}\x_{1:2}\|_2^2$ is now \emph{minimized} (as in Subsection \ref{sec:alghopposit}, $H_{:,1:2}$ stands for the first two columns of $H$ and $\x_{1:2}$ stands for the first two components of $\x$). After $\x_2$ is fixed the algorithm then proceeds by fixing $\x_3$ so that $\|H_{:,1:3}\x_{1:3}\|_2^2$ is \emph{minimized} ($H_{:,1:3}$ stands for the first three columns of $H$ and $\x_{1:3}$ stands for the first three components of $\x$) and so on until one fixes all components of $\x$.

Similarly to what we did when we analyzed the positive counterpart, to analyze the algorithm we will set $\hat{\x}_1=1$, $r_1^{(neg)}=\|H_{:,1}\|_2^2$, and for any $2\leq k\leq n$
\begin{eqnarray}
\hat{\x}_k & = & \mbox{argmin}_{\x_k\in\{-1,1\}}\|H_{:,1:k}\begin{bmatrix} \hat{\x}_{1:k-1}\\ \x_k\end{bmatrix} \|_2^2\nonumber \\
r_k^{(neg)} & = & \min_{\x_k\in\{-1,1\}}\|H_{:,1:k}\begin{bmatrix} \hat{\x}_{1:k-1}\\ \x_k\end{bmatrix} \|_2^2=\|H_{:,1:k}\hat{\x}_{1:k}\|_2^2.\label{eq:defrnegit}
\end{eqnarray}
Our goal will be to compute $Er_n^{(neg)}$. We will do so in a recursive fashion. To that end we will start with $Er_2^{(neg)}$
\begin{multline}
Er_2^{(neg)}=\min_{\x_2\in\{-1,1\}}\|H_{:,1:2}\begin{bmatrix} \hat{\x}_{1}\\ \x_2\end{bmatrix} \|_2^2
=\min_{\x_2\in\{-1,1\}}\|H_{:,1:2}\begin{bmatrix} 1\\ \x_2\end{bmatrix} \|_2^2\\=E\|H_{:,1}\|_2^2+2E(\min_{\x_2\in\{-1,1\}}\x_2(H_{:,2}^TH_{:,1}))+E\|H_{:,2}\|_2^2
=Er_1^{(neg)}-2\sqrt{\frac{2}{\pi}}E\sqrt{r_1^{(neg)}}+m.\label{eq:rnegit1}
\end{multline}
One can then apply a similar strategy to obtain for a general $2\leq k\leq n$
\begin{multline}
Er_k^{(neg)}=\min_{\x_k\in\{-1,1\}}\|H_{:,1:k}\begin{bmatrix} \hat{\x}_{1:k-1}\\ \x_k\end{bmatrix} \|_2^2
=\min_{\x_k\in\{-1,1\}}\|\begin{bmatrix}H_{:,1:k-1} & H_{:,k} \end{bmatrix}\begin{bmatrix} \hat{\x}_{1:k-1}\\ \x_k\end{bmatrix} \|_2^2\\=E\|H_{:,1:k-1}\hat{\x}_{1:k-1}\|_2^2+2E(\min_{\x_k\in\{-1,1\}}\x_k(H_{:,k}^TH_{:,1:k-1}\hat{\x}_{1:k-1}))+E\|H_{:,k}\|_2^2
=Er_{k-1}^{(neg)}-2\sqrt{\frac{2}{\pi}}E\sqrt{r_{k-1}^{(neg)}}+m.\label{eq:rnegit2}
\end{multline}
As earlier, to make the exposition easier we will assume that $n$ is large and switch to the limiting behavior of $Er^{(neg)}$'s. Again, assuming concentration of $r_k^{(neg)}$'s (for $k$ proportional to $n$) around their mean values will then give $\lim_{n\rightarrow\infty} \frac{E\sqrt{r_k^{(neg)}}}{n}=\lim_{n\rightarrow\infty} \frac{\sqrt{Er_k^{(neg)}}}{n}$. One then based on (\ref{eq:defrnegit}), (\ref{eq:rnegit1}), and (\ref{eq:rnegit2}) can establish the following recursion for finding $Er_n^{(neg)}$
\begin{equation}
\phi_k=\phi_{k-1}-2\sqrt{\frac{2}{\pi}}\sqrt{\phi_{k-1}}+m,\label{eq:rnegit3}
\end{equation}
with $\phi_1=m$ and $\lim_{n\rightarrow\infty}\frac{Er_n^{(neg)}}{n^2}=\lim_{n\rightarrow\infty}\frac{\phi_n}{n^2}$. Computing the last limit can then be done to a fairly high precision for any different $m$. Following the example we chose in the positive case, we note that for $m=n$ (i.e. $\alpha=1$) one has
\begin{equation}
\lim_{n\rightarrow\infty}\frac{Er_n^{(neg)}}{n^2}=\lim_{n\rightarrow\infty}\frac{\phi_n}{n^2}\approx .3072.\label{eq:numnegitlb}
\end{equation}
One can also compare this result to the results of the previous section to get
\begin{equation}
\lim_{n\rightarrow\infty}\frac{E\xi_n}{\sqrt{n}}\geq\lim_{n\rightarrow\infty}\frac{E\sqrt{r_n^{(neg)}}}{n}=\lim_{n\rightarrow\infty}\frac{\sqrt{\phi_n}}{n}\approx \sqrt{0.3072}\approx 0.55.\label{eq:numnegitlb1}
\end{equation}
This is substantially away from the lower bound $0.2021$ one would get in Subsection \ref{sec:neghoplb} when $\alpha=1$ (i.e. $m=n$)). However, as was the case with the positive form in earlier sections, the bound given above is algorithmic.


\section{Conclusion}
\label{sec:conc}

In this paper we looked at classic Hopfield forms. We first viewed the standard positive Hopfield form and then defined its a negative counterpart. We were interested in their behavior in the zero-temperature limit which essentially amounts to the behavior of their ground state energies. We then sketched mechanisms that can be used to provide upper and lower bounds for the ground state energies of both models.

To be a bit more specific, we first provided purely theoretical bounds on the expected values of the ground state energy of the positive Hopfield model. These bounds appear to be fairly close to each other (moreover, the upper bounds actually don't even require the thermodynamic regime). In addition to that we also presented two very simple (certainly well known) algorithms that can be used to approximately determine the ground state energy of the positive Hopfield model. For both algorithms we then sketched how one can determine their performance guarantees. As it turned out, these algorithms provide a fairly good approximations (while the analytical results that we provided demonstrated that they are in certain scenarios about $10\%$ away from the optimal values, practically, in these same scenarios, their objective values are not more than $5\%$ away from the optimal value).

We then translated our results related to the positive Hopfield form to the case of the negative Hopfield form. We again targeted the ground state regime and provided a theoretical lower bound for the expected behavior of the ground state energy. We also, showed how one of the algorithms that we designed for the positive form can easily be adapted to fit the negative form. This enabled us to get an algorithmic upper bound for the ground state energy of the negative form. While, the bounds we obtained for the negative form are not as good as the ones we obtained for the positive form, they are obtained in a very simple manner and provide in a way a quick assessment how the ground state energies of these forms behave.

For several results that relate to the behavior of the expected ground state energies, we also showed that the corresponding (more general) probabilistic results hold in the thermodynamic limit.

Moreover, the purely theoretical results we presented are for the so-called Gaussian Hopfield models. Often though a binary Hopfield model may be a more preferred optiion. However, all results that we presented can easily be extended to the case of binary Hopfield models (and for that matter to an array of other statistical models as well). There are many ways how this can be done. Proving that is not that hard. In fact there are many ways how it can be done, but typically would boil down to repetitive use of the central limit theorem. For example, a particularly simple and elegant approach would be the one of Lindeberg \cite{Lindeberg22}. Adapting our exposition to fit into the framework of the Lindeberg principle is relatively easy and in fact if one uses the elegant approach of \cite{Chatterjee06} pretty much a routine. Since we did not create these techniques we chose not to do these routine generalizations. However, to make sure that the interested reader has a full grasp of generality of the results presented here, we do emphasize again that pretty much any distribution that can be pushed through the Lindeberg principle would work in place of the Gaussian one that we used.

We should also mention that the algorithms we presented are simple and certainly not the best known. One can design algorithms that can practically achieve a way better performance for both Hopfield forms. However, since their performance analysis is not easy we leave their detailed exposition for an algorithmic presentation. We do mention though, that out idea here was not to introduce the best possible algorithms but rather to show how one can use the simple ones to get results related to the behavior of the optimal objective value.

It is also important to emphasize that we in this paper presented a collection of very simple observations. One can improve many of the results that we presented here but at the expense of the introduction of a more complicated theory. We will present results in many such directions elsewhere. We do recall though, that in this paper we were mostly concerned with the behavior of the ground state energies. A vast majority of our results can be translated to characterize the behavior of the free energy when viewed at any temperature. However, as mentioned above, this requires a way more detailed exposition and we will present it elsewhere.

\begin{singlespace}
\bibliographystyle{plain}
\bibliography{HopBndsRefs}

\begin{thebibliography}{10}

\bibitem{BarGenGueTan12}
A.~Barra, F.~Guerra G.~Genovese, and D.~Tantari.
\newblock How glassy are neural networks.
\newblock {\em J. Stat. Mechanics: Thery and Experiment}, July 2012.

\bibitem{BarGenGueTan10}
A.~Barra, G.~Genovese, and F.~Guerra.
\newblock The replica symmetric approximation of the analogical neural network.
\newblock {\em J. Stat. Physics}, July 2010.

\bibitem{Chatterjee06}
S.~Chatterjee.
\newblock A generalization of the {L}indenberg principle.
\newblock {\em The Annals of Probability}, 34(6):2061--2076.

\bibitem{GiuGen12}
G.~Genovese.
\newblock Universality in bipartite mean field spin glasses.
\newblock {\em J. Math. Phys.}, 53:123304, 2012.

\bibitem{Gordon88}
Y.~Gordon.
\newblock On {M}ilman's inequality and random subspaces which escape through a
  mesh in ${R}^n$.
\newblock {\em Geometric Aspect of of functional analysis, Isr. Semin. 1986-87,
  Lect. Notes Math}, 1317, 1988.

\bibitem{Guerra03}
F.~Guerra.
\newblock Broken replica symmetry bounds in the mean field spin glass model.
\newblock {\em Comm. Math. Physics}, 233:1--12, 2003.

\bibitem{Hebb49}
D.~O. Hebb.
\newblock Organization of behavior.
\newblock {\em New York: Wiley}, 1949.

\bibitem{Hop82}
J.~J. Hopfield.
\newblock Neural networks and physical systems with emergent collective
  computational abilities.
\newblock {\em Proc. Nat. Acad. Science}, 79:2554, 1982.

\bibitem{Lindeberg22}
J.~W. Lindeberg.
\newblock Eine neue herleitung des exponentialgesetzes in der
  wahrscheinlichkeitsrechnung.
\newblock {\em Math. Z.}, 15:211--225, 1922.

\bibitem{Parisi80}
G.~Parisi.
\newblock Breaking the symmetry in sk model.
\newblock {\em J. Physics}, A13:1101, 1980.

\bibitem{PasFig78}
L.~Pastur and A.~Figotin.
\newblock On the theory of disordered spin systems.
\newblock {\em Theory Math. Phys.}, 35(403-414), 1978.

\bibitem{PasShchTir94}
L.~Pastur, M.~Shcherbina, and B.~Tirozzi.
\newblock The replica-symmetric solution without the replica trick for the
  hopfield model.
\newblock {\em Journal of Statistical Physiscs}, 74(5/6), 1994.

\bibitem{ShchTir93}
M.~Shcherbina and B.~Tirozzi.
\newblock The free energy of a class of hopfield models.
\newblock {\em Journal of Statistical Physiscs}, 72(1/2), 1993.

\bibitem{SheKir72}
D.~Sherrington and S.~Kirkpatrick.
\newblock Solvable model of a spin glass.
\newblock {\em Phys. Rev. Letters}, 35:1792--1796, 1972.

\bibitem{Slep62}
D.~Slepian.
\newblock The one sided barier problem for gaussian noise.
\newblock {\em Bell System Tech. Journal}, 41:463--501, 1962.

\bibitem{Tal98}
M.~Talagrand.
\newblock Rigorous results for the hopfield model with many patterns.
\newblock {\em Prob. theory and related fields}, 110:177--276, 1998.

\bibitem{Tal06}
M.~Talagrand.
\newblock The parisi formula.
\newblock {\em Annals of mathematics}, 163(2):221--263, 2006.

\bibitem{JYZhao11}
J.~Y. Zhao.
\newblock The hopfield model with superlinearly many patterns.
\newblock 2011.
\newblock available online at arXiv:1108.4771.

\end{thebibliography}
\end{singlespace}

\end{document}